\documentclass[11pt,a4paper]{amsart}

\usepackage{amsmath, amsthm, amssymb, graphicx, color, euscript, appendix,enumitem,microtype,hyperref,mathtools}

\newtheorem{theorem}{Theorem}

\newtheorem{proposition}{Proposition}
\newtheorem{corollary}{Corollary}
\newtheorem{lemma}{Lemma}

\theoremstyle{definition}

\theoremstyle{remark}

\newcommand{\R}{\mathbb{R}} 
\newcommand{\C}{\mathbb{C}} 
\newcommand{\Z}{\mathbb{Z}}

\newcommand{\cA}{{\ensuremath{\mathcal{A}}}}

\newcommand{\cM}{{\ensuremath{\mathcal{M}}}}
\newcommand{\cO}{{\ensuremath{\mathcal{O}}}}

\DeclareMathOperator{\M}{M}
\DeclareMathOperator{\GL}{GL}
\DeclareMathOperator{\SL}{SL}

\DeclareMathOperator{\tr}{tr}
\DeclareMathOperator{\RR}{R}

\thanks{The research of the first author was partially supported by Schweizerische Nationalfonds Grant 200021-178730. 
The research of the second author was partially supported by Schweizerische Nationalfonds Grant PP00P2\_157583}
\keywords{products of exponentials, Bass stable rank, rings of holomorphic functions}
\subjclass{15A54, 15A16, 30H50, 32A38, 32E10, 48E25}

\begin{document}
\title{Exponential factorizations of holomorphic maps}
\author{Frank Kutzschebauch and Luca Studer}
\email{frank.kutzschebauch@math.unibe.ch, \ luca.studer@math.unibe.ch}
\begin{abstract} 
We show that any element of the special linear group $\SL_2(\RR)$ is a product of two exponentials if the ring $\RR$ is either 
the ring of holomorphic functions on an open Riemann surface or the disc algebra. 
This is sharp: one exponential factor is not enough since the exponential map corresponding to $\SL_2(\C)$ is not surjective. 
Our result extends to the linear group $\GL_2(\RR)$.
\end{abstract} 

\maketitle

\section{introduction}
For a Stein space $X$, a complex Lie group $G$ and its exponential map $\exp: \mathfrak{g} \to G$ we say that a holomorphic map $f:X \to G$ 
is a product of $k$ exponentials if there are holomorphic maps $f_1, \ldots, f_k:X \to \mathfrak{g}$ such that $$f=\exp(f_1)\cdots \exp(f_k).$$
It is easy to see that any map $f$ which is a product of exponentials (for some sufficiently large $k$) is null-homotopic. 
In the case where $G$ is the special linear group $\SL_n(\C)$ the converse follows from~\cite{IK} as explained in~\cite{DK}. However, it turns out to be a difficult problem to determine the minimal number 
$k$ of needed factors in dependence of the dimensions of $X$ and $\SL_n(\C)$. We solve this problem 
for $\dim X=1$ and $n=2$.

\begin{theorem}
\label{Riemann}
Any holomorphic map from an open Riemann surface to the special linear group $\SL_2(\C)$ is a product of two exponentials.
\end{theorem}

Theorem~\ref{Riemann} improves a result of Doubtsov and Kutzschebauch, who showed the same result with three instead of two factors in the conclusion, see Proposition 3 in~\cite{DK}. Stated differently, Theorem~\ref{Riemann} says that every element of $\SL_2(\cO(X))$ can be written as a product of two exponentials, where $\cO(X)$ denotes the 
ring of holomorphic functions on a given open Riemann surface $X$. Our second result is of similar flavor, but the ring $\cO(X)$ is replaced by the disc algebra $\cA$. By definition, the disc algebra $\cA$ is the $\C$-Banach algebra of continuous functions on the closed 
disc $\{z \in \C: |z|\leq 1\}$ which are holomorphic on the interior, equipped with the supremum norm. 

\begin{theorem}
\label{disc}
For the disc algebra $\cA$, any element of $\SL_2(\cA)$ is a product of two exponentials.
\end{theorem}

Recall that the exponential map $\exp: \mathfrak{sl}_2(\C) \to \SL_2(\C)$ is not surjective. In this sense Theorem~\ref{Riemann} 
and~\ref{disc} are sharp. It is worth mentioning that $\SL_2(\C)$ is simply connected implying that holomorphic maps from open Riemann surfaces to $\SL_2(\C)$ and elements of $\SL_2(\cA)$ are null-homotopic. This is the reason that the map in question being null-homotopic is a redundant assumption in Theorem~\ref{Riemann} and~\ref{disc}. 
As corollaries of Theorem~\ref{Riemann} and~\ref{disc} we get the analogous results if the special linear group is replaced by the linear group with the corresponding entries.

\begin{corollary}
\label{Riemann2}
Any null-homotopic holomorphic map from an open Riemann surface to the linear group $\GL_2(\C)$ is a product of two exponentials. 
\end{corollary}

\begin{proof}
Let $X$ be an open Riemann surface and $\M_2(\C)$ the complex $2 \times 2$-matrices. 
If a given holomorphic map $A:X \to  \GL_2(\C)$ is null-homotopic, then $\det A: X \to \C^\ast$ is null-homotopic as well. Therefore 
$\det A$ has a holomorphic logarithm $\log:X \to \C$, satisfying $e^{\log}=\det A$. In particular, if $D:X \to \M_2(\C)$ is 
the diagonal matrix with diagonal entries $\log/2$, $\exp(-D)A$ has values in $\SL_2(\C)$. 
By Theorem~\ref{Riemann} there are holomorphic $B, C: X \to \M_2(\C)$ such that 
$$A=e^De^{-D}A=e^De^Be^C=e^{D+B}e^C,$$ where we used in the last equality that $D$ commutes with all other matrices. This finishes the proof.
\end{proof}
Unlike in Theorem~\ref{Riemann}, in Corollary~\ref{Riemann2} the assumption that $f$ is null-homotopic is not redundant.
For instance, 
\begin{align*} A(z)=
\begin{pmatrix}
z & 0 \\
0 & z
\end{pmatrix}, \ z \in \C^\ast
\end{align*}
is not null-homotopic since otherwise $\det A: \C^\ast \to \C^\ast, z \mapsto z^2$ would be null-homotopic as well.

\begin{corollary}
\label{disc}
For the disc algebra $\cA$, any element of $\GL_2(\cA)$ is a product of two exponentials.
\end{corollary}

\begin{proof}
This follows from Theorem~\ref{disc} in the same way as Corollary~\ref{Riemann2} follows from Theorem~\ref{Riemann}. Here, we need in addition that any unit in $\cA$ has a logarithm, which follows from the fact that the disc (and thereby the domain of the elements of $\cA$) is contractible. In particular, the map in question being null-homotopic is again a redundant assumption.
\end{proof}

Corollary~\ref{disc} improves a result of Mortini and Rupp, who showed the same with four instead of two factors in the conclusion, see Theorem 7.1 in~\cite{MR}. Also Corollary~\ref{Riemann2} and~\ref{disc} are sharp in the sense that 
one exponential factor is not enough. An example is the matrix
\begin{align*}
A(z)= 
\begin{pmatrix}
1 & 1\\
0 & e^{4\pi i z}
\end{pmatrix}, z \in \Delta.
\end{align*}
One can show that the second entry of any lift of $z \mapsto A(z)$, $|z|<1/2$ via the exponential map tends to infinity if $z \to 1/2$. For details see~\cite{MR}, Example~6.4. 

We would like to thank Sebastian Baader for helpful comments on a draft of this text.

\section{Proof of Theorem~\ref{Riemann}}
An important ingredient in the proof is an Oka principle due to Forstneri\v c, which follows essentially from Theorem 2.1 in~\cite{Forstneric}. The version, which we use in this text is the below stated Theorem~\ref{Oka principle 2}. It is used to show Proposition~\ref{vanishing trace}, which is the main ingredient in the proof of Theorem~\ref{Riemann}.
Throughout this section $X$ denotes an open Riemann surface. 

\begin{proposition}
\label{vanishing trace}
Let $A: X \to \SL_2(\C)$ be holomorphic and assume that $A(x)$ has distinct eigenvalues for some $x \in X$. Then 
$A=BC$ for suitable holomorphic $B,C: X \to \SL_2(\C)$, both of which have vanishing trace.
\end{proposition}

Note that the conclusion of Proposition~\ref{vanishing trace} is equivalent to finding a holomorphic $B:X \to \SL_2(\C)$ such that 
$B$ and $AB$ have vanishing trace, simply since taking the inverse of a $2\times 2$-matrix with trace zero has again trace zero. 
Expressed differently, Proposition~\ref{vanishing trace} is proved if we can show the existence of a global section of the bundle
\begin{align*}
Z\coloneqq \{(x,B) \in X \times \SL_2(\C): \tr(B)=\tr(A(x)B)=0\}
\end{align*}
over $X$. If $a,b,c,d$ denote the coefficients of $A$, and $u,w,v,-u$ denote the coefficients of $B$, we can express $Z$ more explicitly as 
\begin{align*}
\{(x,u,v,w) \in X \times \C^3: (a(x)-d(x))u+b(x)v+c(x)w=0, \ u^2+vw=-1\}.
\end{align*}

More concretely, Proposition~\ref{vanishing trace} is proved if we manage the prove the following reformulation.

\begin{proposition}
\label{section}
Let $A: X \to \SL_2(\C)$ be holomorphic and assume that $A(x)$ has distinct eigenvalues for some $x \in X$. Then 
the restriction $h$ of the projection $X \times \C^3 \to X$ to $Z$ has a holomorphic section.
\end{proposition}

For an open subset $U\subset X$, $Z|U$ denotes the restriction of the bundle $h: Z\to X$ to $h^{-1}(U)$.
We start the proof of Proposition~\ref{section} with the following simple

\begin{lemma}
\label{local}
For every $x \in X$ there is a neighborhood $U$ of $x$ and a holomorphic section $F:U \to Z|U$ of $Z|U$.
\end{lemma}

\begin{proof}
After passing to a local chart we may assume that $X$ is the unit disc $\Delta \coloneqq \{z \in \C: |z|<1\}$ and $x=0$. Finding a local holomorphic 
section in a neighborhood of $0$ is equivalent to finding a neighborhood $0 \in U\subset \Delta$ and holomorphic maps $u,v,w:U \to \C$, which 
satisfy 
\begin{align}
\label{1}
(a-d)u+bv+cw=0, \ \ u^2+vw=-1.
\end{align}
Local holomorphic solutions to~(\refeq{1}) exist if and only if there are local holomorphic solutions to the less restrictive problem
\begin{align}
\label{2}
(a-d)u+bv+cw=0, \ \ u^2+vw \in \cO^\ast_0.
\end{align}
The reason is that if $u,v,w$ are local solutions in a neighborhood of the origin to~(\refeq{2}), we can rescale these solutions with a local holomorphic square root of $u^2+vw$, or more precisely, by defining new solutions by $\tfrac{iu}{r}, \tfrac{iv}{r}, \tfrac{iw}{r}$ for some $r:U \to \C^\ast$ satisfying $r^2=u^2+vw$ defined on a sufficiently small neighborhood $U$ of the origin. 
To find solutions to~(\refeq{2}) we distinguish three cases. Let $n(f) \in \Z_{\geq 0}$ denote the vanishing order of a holomorphic function $f:\Delta \to \C$ at the origin. 
The first case is $n(a-d)\geq n(b)$. Then $-\tfrac{a-d}{b}$ is holomorphic in a neighborhood of $0$ and $u=1$, $v=-\tfrac{a-d}{b}$ and $w=0$ is a solution to~(\refeq{2}). 
The second case $n(a-d)\geq n(c)$ we find similarly a solution $u=1$, $v=0$ and $w=-\tfrac{a-d}{c}$ to~(\refeq{2}). The remaining case is $n(a-d)<\min(n(b),n(c))$, 
which implies $n(a-d)<n(b+c)$ and hence $-\tfrac{b+c}{a-d}$ is holomorphic in a neighborhood of the origin and vanishes at the origin. Then $u=-\tfrac{b+c}{a-d}$, $v=1$, $w=1$ solves~(\refeq{2}). This finishes the proof.
\end{proof}

Let $D$ denote the discriminant of $A$, that is $D\coloneqq (a+d)^2-4$. By \textit{isomorphic} fiber bundles we mean isomorphic as complex analytic fiber bundles.

\begin{lemma}
\label{local trivialization}
Let $U\subset X \setminus (\{D=0\}\cup \{c=0\})$ be an open neighborhood where $D:U \to \C$ has a holomorphic square root $\sqrt D$, and 
set $f\coloneqq \tfrac{d-a + \sqrt{D}}{2c}.$ Then $Z|U$ is isomorphic to $U \times \C^\ast$, and an isomorphism is given by 
$$\phi:Z|U \to U \times \C^\ast, \ \phi(x,u,v,w)=(x, u+f(x)v).$$
\end{lemma}

\begin{proof}
First we do the necessary computations at the level of a single fiber. 
For this, we think of the coefficients $a,b,c,d$ of $A$ as elements of $\C$. We want to determine all $u,v,w \in \C$ such that 
\begin{align*}
(a-d)u+bv+cw=0, \ \ -u^2-vw=1.
\end{align*}
Since $c\not = 0$, we can solve for $w$ and get equivalently 
\begin{align*}
-1		&=u^2+vw \\
		&=u^2+v\tfrac{(d-a)u-bv}{c} \\
		&=u^2+\tfrac{d-a}{c}uv-\tfrac{b}{c}v^2 \\
		&=\Big(u+\tfrac{d-a}{2c}v\Big)^2-\Big(\tfrac{(d-a)^2}{4c^2}+\tfrac{b}{c}\Big)v^2.
\end{align*}
Furthermore we have 
\begin{align*}
\frac{(d-a)^2}{4c^2}+\frac{b}{c}=\frac{(d+a)^2-4ad}{4c^2}+\frac{4bc}{4c^2}=\frac{(d+a)^2-4(ad-bc)}{4c^2}=\frac{D}{4c^2}.
\end{align*}
Fix a square root $\sqrt D$ of $D$ and note that 
$$\tilde u=u+\tfrac{d-a+\sqrt{D}}{2c}v, \ \ \tilde v=u+\tfrac{d-a-\sqrt{D}}{2c}v$$ defines a linear coordinate change of $\C^2$, which translates 
the above equation to
\begin{align*}
-1		&=\Big(u+\tfrac{d-a}{2c}v\Big)^2-\tfrac{D}{4c^2}v^2 \\
		&=\Big(u+\tfrac{d-a}{2c}v\Big)^2-\Big(\tfrac{\sqrt{D}}{2c}v\Big)^2 \\
		&=\Big( u+\tfrac{d-a+\sqrt{D}}{2c}v\Big ) \Big(u+\tfrac{d-a-\sqrt{D}}{2c}v \Big) \\
		&=\tilde u \tilde v.
\end{align*}
This shows that the fiber is given by $\{(\tilde u, \tilde v) \in \C^2: \tilde u \tilde v=-1\}=\C^\ast$ and that $(u,v,w) \to u+\tfrac{d-a+\sqrt{D}}{2c}v$ is an 
isomorphism of the fiber onto $\C^\ast$. Moreover, our computations yield a trivialization of $Z|U$, which is defined similarly, or more precisely, as in the 
assumption of the Lemma. This is the case since our computations work out just the same way if we have a holomorphic dependence on $x \in U$.
\end{proof}

\begin{lemma}
\label{D}
Over $X\setminus \{D=0\}$, $h:Z \to X$ is a fiber bundle with fiber $\C^\ast$.
\end{lemma}

\begin{proof}
At points $x \in X\setminus \{D=0\}$ with $c(x)\not =0$, choose a neighborhood $U\subset X$ of $x$ such that $c|U$ does not 
vanish, and such that $D$ has a square root on $U$. Then a trivialization of $Z|U$ is given by Lemma~\ref{local trivialization}. 
In the case $c(x)=0$, let us reduce the problem to the case $c(x)\not =0$ with the following observation. Our bundle is given by 
\begin{align*}
Z=\{(x,B) \in X \times \SL_2(\C): \tr(B)=\tr(A(x)B)=0\}.
\end{align*}
Define for $P \in \SL_2(\C)$ a bundle 
\begin{align*}
Z_P		&=\{(x,PBP^{-1}) \in X \times \SL_2(\C): \tr(B)=\tr(A(x)B)=0\}.
\end{align*}
Clearly $Z$ and $Z_P$ are isomorphic over $X$. Since conjugation with a matrix does not change the trace, we obtain with the substitution $C=PBP^{-1}$ 
\begin{align*}
Z_P		&=\{(x, C) \in X \times \SL_2(\C): \tr(P^{-1}CP)=\tr(A(x)P^{-1}CP)=0\} \\
		&=\{(x, C) \in X \times \SL_2(\C): \tr(C)=\tr(PA(x)P^{-1}C)=0\}.
\end{align*}
Note that if the third entry $c$ of $A$ equals $0$ at $x$, then, since $D(x)\not =0$ and hence $A(x)\not = \pm id$, there is $P \in \SL_2(\C)$ such that 
the third entry of $PA(x)P^{-1}$ does not vanish. Using that $Z$ and $Z_P$ are isomorphic and that we can solve the problem for $Z_P$ close to 
$x$, the statement follows. 
\end{proof}

To finish the proof of Propostion~\ref{section} we need the following special case of Theorem 6.14.6, p.\,310 in~\cite{Francs book}.

\begin{theorem}
\label{Oka principle 2}
Let $h:Z \to X$ be a holomorphic map of a reduced complex space $Z$ onto a reduced Stein space $X$. Let $X'\subset X$ be a 
complex analytic subvariety and let $Z'\coloneqq h^{-1}(X')$ and assume that the restriction $h:Z\setminus Z' \to X \setminus X'$ is an elliptic submersion. 
Moreover, let $f:X \to Z$ be a continuous section of $h$ which is holomorphic in a neighborhood of $X'$. Then $f$ is homotopic through 
continuous sections of $h$ which are holomorphic in a fixed small neighborhood of $X'$ to a holomorphic section of $h$.
\end{theorem}

A consequence of this is the following 

\begin{proposition}
\label{Oka principle}
Let $h:Z \to X$ be a holomorphic map from a reduced complex space onto an open Riemann surface. Moreover, 
assume that there is a discrete set $X'\subset X$ such that for $Z'=h^{-1}(X')$, the restriction $h:Z \setminus Z' \to X \setminus X'$ is 
a fiber bundle with fiber $\C^\ast$ and assume that there is a local holomorphic section in a neighborhood of every point of $X'$. 
Then $h$ has a global holomorphic section $f:X \to Z$.
\end{proposition}

\begin{proof}
First we show the existence of a continuous section which is holomorphic in a neighborhood $U$ of $X'$. 
By assumption there is a local holomorphic section $f:U \to Z$ of $h$ defined on a neighborhood $U$ of $X'$. By possibly shrinking 
$U$ we may assume that every connected component of $U$ contains exactly one point of $X'$ and is homeomorphic to a disc, and that 
$f$ extends continuously to $\overline{U}$. $X\setminus X'$ is an open Riemann surface and thus deformation retracts onto a $1$-dimensional CW-complex $K$, 
see e.g.\,\cite{Hamm}. After possibly modifying a fixed deformation retract $r$ of $X \setminus X'$ onto $K$ by a conjugation with a suitable homeomorphism of $X \setminus X'$ we can assume that $\partial U \subset K$. 
Since the fiber $\C^\ast$ of $Z$ is connected we can extend $f|\partial U$ to a section $\tilde f:K \to Z|K$. Since 
$K$ is a deformation retract of $X \setminus X'$ and $h:Z\setminus Z' \to X \setminus X'$ is a fiber bundle, the section $\tilde f$ extends to a continuous section $F:X \setminus X' \to Z \setminus Z'$, see e.g.\,Theorem 7.1, p.\,21 in~\cite{Husemoller}. 
Since $f$ and $F|X \setminus U$ agree on $\partial U$, these two sections define a continuous section $X \to Z$ which agrees with the holomorphic section $f$ on the neighborhood $U$ of $X'$.
The existence of a global holomorphic section follows now from the above Oka principle due to Forstneri\v c, see Theorem~\ref{Oka principle 2}. This finishes the proof.
\end{proof}

\begin{proof}[Proof of Proposition~\ref{section}]
Let $h:Z\to X$ be the bundle over $X$ from Proposition~\ref{section}. 
With Lemma~\ref{local} we proved that there are local sections of $h$ at every point $x \in X$, in particular also at points of the discrete set $X'=\{D=0\}$. Moreover, 
with Lemma~\ref{D} we showed that $h$ is a locally trivial $\C^\ast$-bundle over $X\setminus \{D=0\}$. It follows now from Proposition~\ref{Oka principle} 
that there is a holomorphic section of $h$. This finishes the proof.
\end{proof}

\begin{lemma}
\label{log vanishing}
Let $X$ be an open Riemann surface and let $A:X \to \SL_2(\C)$ be holomorphic with vanishing trace. 
Then $A=e^B$ for some holomorphic $B: X \to \M_2(\C)$ with vanishing trace.
\end{lemma}

\begin{proof}
The characteristic polynomial of $A$ equals $T^2+1$. In particular $\pm i$ are the eigenvalues (at every point $x \in X$). 
There are line bundles $E(i)$ and $E(-i)$ over $X$, whose non-vanishing sections correspond to holomorphic eigenvectors 
of $i$ and $-i$ respectively. Explicitly, we have 
\begin{align*}
E(i) & \coloneqq \{(x,z) \in X \times \C^2: A(x)z=iz\}, \\  
E(-i) & \coloneqq \{(x,z) \in X \times \C^2: A(x)z=-iz\}.
\end{align*}
Since every line bundle over an open Riemann surface is trivial, we have $E(i) \cong X \times \C \cong E(-i)$ as complex analytic line bundles. 
This implies that there are two holomorphic eigenvectors $v: X \to E(i)$, $w:X \to E(-i)$ with $v(x)\not = 0 \not =w(x)$ for all $x \in X$. In 
particular $$P:X \to \M_2(\C), \ P(x)\coloneqq (v(x) \  w(x))$$ takes values in $\GL_2(\C)$ since $v(x)$ and $w(x)$ are eigenvectors of $A(x)$ to 
the distinct eigenvalues $\pm i$. This implies that $A$ is holomorphically diagonalisable with 
\begin{align*}
A=PDP^{-1}, \ \ D\coloneqq 
\begin{pmatrix}
i & 0 \\
0 & -i
\end{pmatrix}.
\end{align*}
For the diagonal matrix $\tilde D$ with entries $ \pm \tfrac{i\pi}{2}$ we have $e^{\tilde D}=D$. We get for $B \coloneqq P \tilde D P^{-1}$ the equality
\begin{align*}
A=PDP^{-1}=Pe^{\tilde D}P^{-1}=e^{P\tilde D P^{-1}}=e^B, 
\end{align*}
as desired. Note that $B$ has vanishing trace since $\tilde D$ has vanishing trace. This finishes the proof.
\end{proof}

\begin{proof}[Proof of Theorem~\ref{Riemann}]
Let $X$ be an open Riemann surface and let $A:X \to \SL_2(\C)$ be a holomorphic map. If the characteristic polynomial of 
$A$ equals $(T-1)^2$, then, since $(A-id)^2=\chi_A(A)=0$ by Cayley-Hamilton, we have 
$$\exp(A-id)=id +(A-id) =A.$$ Moreover, the trace of $A$ is equal to minus the second coefficient of the characteristic polynomial, which implies in our case that $\tr(A-id)=0$, as desired. This shows that $A$ can be written as a single exponential factor. 
If the characteristic polynomial is $(T+1)^2$, then 
the characteristic polynomial of $-A$ is $(T-1)^2$ and since $-id$ is equal to the exponential of the diagonal matrix with diagonal entries 
$\pi i$ and $-\pi i$, $A$ is a product of at most two exponentials with vanishing trace. Otherwise there is $x \in X$ such that $A(x)$ has distinct eigenvalues. 
In that case it follows from Proposition~\ref{vanishing trace} that $A=BC$ for holomorphic $B,C:X \to \SL_2(\C)$ with vanishing trace. In particular, 
the characteristic polynomials of $B$ and $C$ are both $(T-i)(T+i)$. Since $B$ and $C$ have a logarithm by Lemma~\ref{log vanishing}, we are done.
\end{proof}

\section{Proof of Theorem~\ref{disc}}
The proof depends essentially on three ingredients. The first ingredient is that the Bass stable rank of the disc algebra $\cA$ equals $1$. This is needed to reduce the problem to matrices with an invertible first 
entry. The second and third ingredient are the simple facts that the elements of $\cA$ are bounded, and 
that $\exp: \cA \to \cA$ is onto to units of $\cA$. 
In the following $\overline \Delta\subset \C$ denotes the closed unit disc centered at the origin. We use the following notation. If $f:\overline \Delta \to \C$ is a function, then $|f|:\overline \Delta \to \R$ denotes the absolute value $z \mapsto |f(z)|$. In particular, 
the symbol $|f|$ should not be confused with the sup-norm on $\cA$, which is not used explicitly in the proof. Moreover, for $f,g: \overline \Delta \to \R$ we write $f>g$ if $f(z)>g(z)$ for all $z \in \overline \Delta$. The proof depends on the following elementary lemma.

\begin{lemma}
\label{square root}
Let $f \in \cA$ be such that $|f|>2$. Then the polynomial $T^2-fT+1$ has roots $\lambda, \lambda^{-1} \in \cA$ such that 
$|\lambda|>1$.
\end{lemma}

\begin{proof}
First note that our assumption implies that the discriminant $f^2-4$ does not vanish. Therefore 
$f^2-4$ has a square root in $\cA$, which implies that there are roots $\lambda, \lambda^{-1} \in \cA$ of 
$T^2-fT+1$. We have to show that one of $|\lambda|$ and $|\lambda^{-1}|$ is strictly larger than $1$. 
Note that if $T^2-zT+1$, $z \in \C$ has a root $r \in \C$ with $|r|=1$, then we get $|z|=|r^2+1|/|r|=|r^2+1|\leq 2$. Expressed differently, if $|z|>2$, then 
$T^2-zT+1$ has no root on the unit circle. This implies that $\lambda$ and $\lambda^{-1}$ avoid the unit circle, and moreover -- by continuity of $\lambda$ and $\lambda^{-1}$ --  that exactly one of the two is strictly bigger than $1$ in absolute value.
\end{proof}

\begin{proof}[Proof of Theorem~\ref{2}]
Let $$A=\begin{pmatrix}
a & b \\
c & d
\end{pmatrix}\in \SL_2(\cA).$$
It is well-known that the 
Bass stable rank of $\cA$ equals $1$, see~\cite{JMW}. By definition of the Bass stable rank this means that for 
any pair $f,g \in \cA$ with $f \cA + g \cA =\cA$, there is $h \in \cA$ such that 
$f +h g$ is a unit in $\cA$. In particular, since $ad-bc=1$, there is $h \in \cA$ such that 
$a+h c=1$. Consequently the first entry of 
\begin{align*}
\begin{pmatrix}
1 & h \\
0 & 1
\end{pmatrix}
\begin{pmatrix}
a & b \\
c & d
\end{pmatrix}
\begin{pmatrix}
1 & -h \\
0 & 1
\end{pmatrix}=
\begin{pmatrix}
a+h c & \ast \\
\ast & \ast
\end{pmatrix}
\end{align*}
is a unit. Since conjugation with matrices in $\GL_2(\cA)$ does not change the number of 
needed exponential factors to represent a given matrix, this shows that it suffices to consider the case where the first entry $a$ of $A$ is a unit. 
For such $A$, the strategy is as follows: for $\delta>0$ set 
$$B \coloneqq 
 \begin{pmatrix}
\delta  & 0 \\
0 & 1/\delta
\end{pmatrix} 
\begin{pmatrix}
a & b \\
c & d
\end{pmatrix}
= \begin{pmatrix}
\delta a & \delta b \\
c/\delta & d/\delta
\end{pmatrix} \in \SL_2(\cA).$$ If we find $\delta$ such that $B=B(\delta)$ has a logarithm, then -- since $A$ is the product of the diagonal matrix with entries 
$1/\delta, 0, 0, \delta$ and $B$ -- we know that $A$ is a product of two exponentials. Our claim is that $B$ has a logarithm for any sufficiently large $\delta>0$. To see this, let $\delta \geq 1$ be an upper bound of the (bounded) function
$$\beta=\frac{3+|d|}{|a|}.$$ 
From the fact that $\delta \geq 1$ is an upper bound of $\beta$ it follows that 
$$|\tr(B)|= |\delta a +d/\delta| \geq \delta|a|-\frac{|d|}{\delta} \geq  (3+|d|)-|d|>2.$$
By Lemma~\ref{square root} we know that the characteristic polynomial $\chi_B=T^2-\tr(B) T+1$ has roots $\lambda, \lambda^{-1} \in \cA$ with $|\lambda|>1$. 
Since $\lambda$ is a unit in $\cA$, the matrix $D$ with diagonal entries $\lambda$ and $\lambda^{-1}$ has a logarithm given by the diagonal matrix with diagonal entries $\log(\lambda) \in \cA$ and $-\log(\lambda) \in \cA$ for some fixed 
logarithm of $\lambda$. 
Moreover, since conjugation with an element in $\GL_2(\cA)$ does not change the number of needed exponential factors, it suffices to find $P \in \GL_2(\cA)$ with 
$$B=PDP^{-1}.$$
Our claim is that 
$$P=
\begin{pmatrix}
d/\delta-\lambda &-\delta b \\
 -c/\delta & \delta a-\lambda^{-1}
\end{pmatrix} \in \M_2(\cA)$$ does the job. To show this it suffices to show that the columns $v$ resp.\,$w$ of $P=(v \ w)$ satisfy $(B-\lambda id)v=(B-\lambda^{-1} id)w=0$ and that 
$|\det B|\geq 1$. For the first part we get 
$$
(B-\lambda id)v=
\begin{pmatrix}
\delta a -\lambda & \delta b \\
c/\delta & d/\delta -\lambda 
\end{pmatrix}
\begin{pmatrix}
d/\delta-\lambda \\
 -c/\delta 
\end{pmatrix}
= 
\begin{pmatrix}
\chi_B(\lambda) \\
0
\end{pmatrix} =0,
$$
and similarly 
$$(B- \lambda^{-1} id)w=
\begin{pmatrix}
\delta a -\lambda^{-1} & \delta b \\
c/\delta & d/\delta -\lambda^{-1} 
\end{pmatrix}
\begin{pmatrix}
-\delta b \\
\delta a-\lambda^{-1}
\end{pmatrix}
=
\begin{pmatrix}
0 \\
\chi_B(\lambda^{-1})
\end{pmatrix} =0
.$$
For the second part, we get with $ad-bc=1$ $$\det P=-\delta \lambda  a - \delta^{-1} \lambda^{-1} d  +2.$$ 
It follows from $|\lambda|>1$ that
$$|\det P| \geq  \delta |\lambda| |a|-\delta^{-1} |\lambda^{-1}| |d|-2 \geq  \delta |a|-\delta^{-1}|d|-2.$$ 
Furthermore, the fact that $\delta \geq 1$ bounds $\beta=(3+|d|)/|a|$ from above yields
$$\delta |a|-\delta^{-1}|d|-2 \geq (3+|d|)-|d|-2=1,$$ which shows that $|\det P|\geq 1$. This finishes the proof.
\end{proof}

\end{document}